\documentclass[]{elsarticle}
\usepackage{amssymb, amsmath, mathrsfs, amsthm, enumerate, hyperref, mathtools}

\numberwithin{equation}{section}
\newcommand\scalemath[2]{\scalebox{#1}{\mbox{\ensuremath{\displaystyle #2}}}}

\theoremstyle{definition}
\newtheorem{define}{Definition}[section]
\newtheorem{thm}[define]{Theorem}
\newtheorem{lem}[define]{Lemma}
\newtheorem{rmk}[define]{Remark}
\newtheorem{prop}[define]{Proposition}
\newtheorem{cor}[define]{Corollary}

\begin{document}

\begin{frontmatter}

\title{Nice elliptic curves}
\author{Duc Van Huynh}
\ead{dhuynh@georgiasouthern.edu}
\address{Department of Mathematical Sciences, Georgia Southern University, Savannah, GA 31419, USA}

\date{\today}

\begin{abstract}
In this work, we study elliptic curves of the form $E_L: y^2 = x(x-1)(x-L)$, where $L^2-L+1 \in \left(\mathbb{Q}^{\times}\right)^2$ and $L \in \mathbb{Q} \setminus \{0,1\}$. We will show that for almost all quadratic extensions $K/\mathbb{Q}$ except one we have $E_L(K) \cong \mathbb{Z}/2\mathbb{Z} \oplus \mathbb{Z}/2\mathbb{Z}$.
\end{abstract}

\end{frontmatter}

\section{Introduction and results}
There is a $6$ to $1$ Galois cover $\lambda: X(2) \rightarrow X(1)$ of modular curves ramifying above $j=0$ and $j=1728$ \cite[Proposition III.1.7]{silverman1}, where $\lambda$ is the modular lambda function. Each point of $\left(\mathrm{im}\, \lambda \cap \mathbb{Q} \right)$ corresponds to a family of isomorphic elliptic curves with the $\mathbb{Q}$-rational points containing a copy of Klein Four. However, we show that if further that $\lambda^2 - \lambda + 1 \in \mathbb{Q}^2$, then the associated elliptic curves possess much more explicit arithmetic data.

We say that a degree $3$ monic polynomial $f(x)$ in $\mathbb{Q}[x]$ is nice if $f(x)$ and its derivative $f'(x)$ split completely as linear factors in $\mathbb{Q}[x]$. We will assume that each nice polynomial $f(x)$ has distinct roots. Such polynomials have been completely described (see \cite{bd}). For nice polynomials of higher degree with distinct roots, see \citep{ajai}. In this paper, we study elliptic curves of the form
\begin{equation}
E_L: y^2 = f_L(x)=x(x-1)(x-L),
\end{equation}
where $f_L(x)$ is a nice polynomial with distinct roots. Since $f_L(x)$ has distinct roots, $E$ is nonsingular.  Note that $f'(x)$ splits completely over $\mathbb{Q}$ exactly when $L^2 - L + 1 \in \mathbb{Q}^2$. Henceforth, we will call $E_L$ a \textit{nice elliptic curve} if $L \in \mathbb{Q} - \{0,1\}$ such that $L^2 - L + 1 \in \mathbb{Q}^2$.

The results of this paper are stated in the theorem below.

\begin{thm} \label{main} Let $E_L: y^2 = f_L(x)$ be a nice elliptic curve, where $f_L$ is a nice polynomial of the form $f_L(x) = x(x-1)(x-L)$. Let $r$ and $s$ be the two roots of the derivative $f'(x)$ with $r < s$. Then we have the following:
\begin{enumerate}[(i)]
\item The curve $E_L$ is a nice elliptic curve exactly when 
\[L = \dfrac{t^2-1}{2t-1},\]
where $t \in \mathbb{Q} - \{0, 1/2, 2, \pm 1\}$.
\item For a quadratic extension $K/\mathbb{Q}$, the torsion subgroup $T_L(K)$ of $E_L(K)$ is isomorphic to
\[ T_L(K) \cong \begin{cases} 
      \mathbb{Z}/2\mathbb{Z} \oplus \mathbb{Z}/2\mathbb{Z} & \text{ for } K \neq \mathbb{Q}(\sqrt{1-L}) \\
      \mathbb{Z}/4\mathbb{Z} \oplus \mathbb{Z}/2\mathbb{Z} & \text{ for } K = \mathbb{Q}(\sqrt{1-L}) 
   \end{cases}
\]

\item If $L$ is defined by 
\[L = \dfrac{(u^2+3)(u^2-1)}{4u^2},\]
where $u \in \mathbb{Q} - \{0, \pm 1\}$, then $(s, \sqrt{f_L(s)})$ is a point of infinite order.
\item With $L$ defined as in part (ii), the rank of $E_L(\mathbb{Q}(\sqrt{-3}))$ is at least $2$, where two of the linearly independent points of infinite order are $(r,\sqrt{f_L(r)})$ and $(s,\sqrt{f_L(s)})$.
\end{enumerate}
\end{thm}

To show that the torsion subgroup of $E_L(\mathbb{Q})$ is isomorphic $\mathbb{Z}/2\mathbb{Z} \oplus \mathbb{Z}/2\mathbb{Z}$, it suffices to show that $E_L(\mathbb{Q})$ lacks non-trivial $4$-torsion and $3$-torsion. Showing that it lacks non-trivial $4$-torsion is straightforward; however, showing that it lacks non-trivial $3$-torsion requires Chabauty-Coleman's Method. We show that $E_L(\mathbb{Q})$ does not have a point of order $3$ by showing that the $3$-divisional polynomial does not have a rational root, which is equivalent to showing that a certain hyperelliptic curve of genus $3$ has only $12$ points. Admittedly, this route seems adhoc but an injection of torsion into finite fields does not seem feasible.

We then show that $E_L$ does not acquire more points when extended to any quadratic extension over $\mathbb{Q}$ except one. As all torsion points of $E_L(\mathbb{Q})$ are of the form $(t,0)$ for some integer $t$, points of infinite order can be readily found, that is, points of infinite order can be constructed via a quadratic twist or by extending to a quadratic extension. We will focus on the points of infinite order given by the roots of $f'(x)$ as they always live in $\mathbb{Q}(\sqrt{-3})$.

\section{Overview}
In Section \ref{rational}, we will describe how to explicitly construct nice elliptic curves. Then we prove that $E_L(\mathbb{Q})$ does not have a point of order $4$ or $3$. In Section \ref{chab}, we will provide an overview of Chabauty-Coleman's Method. In Section \ref{proof}, we will use Chabauty-Coleman's Method to complete the proof from Section \ref{rational} that $E_L(\mathbb{Q})$ does not have a point of order $3$. Finally, in Section \ref{galois} we will provide some words on Galois representation of nice elliptic curves.

\section{Rational points} \label{rational}
Let $f_L(x) = x(x-1)(x-L)$ be a nice cubic polynomial with distinct roots. As $f_L(x)$ has distinct roots, $L \neq 0,1$. Furthermore, for $f'(x)$ to split over $\mathbb{Q}$, it must be that $L^2 - L + 1 \in \mathbb{Q}^2$. We can explicitly parametrize $L$.

\begin{prop} \label{L} Let $V = \{(L,W): L^2 - L + 1 = W^2 \text{ and } L,W \in \mathbb{Q}\}$. The map $\varphi: \mathbb{Q} -\{1/2\} \rightarrow V$ given by
\begin{equation}
\varphi(s) = \left(\dfrac{s^2-1}{2s-1},\dfrac{s^2-s+1}{2s-1}\right)
\end{equation}
is a one-to-one correspondence.
\end{prop}
\begin{proof}
The variety $V$ is a genus $0$ curve with the rational point $P_0=(0,1)$. Projecting away from $P_0$ gives the desired correspondence.
\end{proof}

Recall that we say $E_L$ is a nice elliptic curve if $E_L$ is defined by
\begin{equation}
E_L: y^2 = f_L(x) = x(x-1)(x-L),
\end{equation}
where $L^2 - L + 1 \in \mathbb{Q}^2 - \{0,1\}$. With Proposition \ref{L}, we can describe $E_L$ explicitly.
\begin{lem} The curve $E_L$ is a nice elliptic curve exactly when 
\[L = \dfrac{s^2-1}{2s-1}\]
where $s \neq 0,1/2, 2,\pm 1$.
\end{lem}

Since the torsion subgroup $T_L(\mathbb{Q})$ of $E_L(\mathbb{Q})$ contains a subgroup isomorphic to $\mathbb{Z}/2\mathbb{Z} \oplus \mathbb{Z}/2\mathbb{Z}$, Mazur's Theorem (\cite{silverman1}, Theorem 7.5) tells us that $T_L(\mathbb{Q})$ is isomorphic to $\mathbb{Z}/2\mathbb{Z} \oplus \mathbb{Z}/2N\mathbb{Z}$, where $N=1,2,3$ or $4$. We will show that $T_L(\mathbb{Q})$ is isomorphic to $\mathbb{Z}/2\mathbb{Z} \oplus \mathbb{Z}/2\mathbb{Z}$ by showing that $T_L(\mathbb{Q})$ does not have a point of order $3$ or $4$. All of the calculations in this section are done using SageMath \cite{sagemath}.

\begin{thm} \label{torsion} For any quadratic extension $K/\mathbb{Q}$ not equal to $\mathbb{Q}(\sqrt{1-L})$, the torsion subgroup $T_L(K)$ of $E_L(K)$ is isomorphic to $\mathbb{Z}/2\mathbb{Z} \oplus \mathbb{Z}/2\mathbb{Z}$, and $T_L(K)$ is isomorphic to $\mathbb{Z}/2\mathbb{Z} \oplus \mathbb{Z}/4\mathbb{Z}$ otherwise.
\end{thm} 

To prove Theorem \ref{torsion}, we will first show that $T_L(K)$ does not have a point of order $4$ or an element of order $3$ over any quadratic extension $K/\mathbb{Q}$ for which $x^2 - (1-L)$ does not split. Then we show the same is true for any quadratic twist $E^D_L$ of $E_L$. Theorem \ref{torsion} follows from the fact that for odd $n$, elements of $E_L(K)[n]$ come from either $E_L(\mathbb{Q})[n]$ or its twist $E_L^D(\mathbb{Q})[n]$.

\begin{lem} \label{4t1}
Let $L \in \mathbb{Q} - \{0,1\}$. If $L^2 - L + 1 \in \mathbb{Q}^2$, then neither $L$ or $1-L$ is a square in $\mathbb{Q}$.
\end{lem}

\begin{proof}
Let $E_1$ and $E_2$ be elliptic curves over $\mathbb{Q}$ defined by
\[E_1: y^2 = x(x^2-x+1)\]
\[E_2: y^2 = (1-x)(x^2-x+1)\]
The rational points of $E_1$ and $E_2$ are 
\[E_1(\mathbb{Q}) = \{(0,0), (1,-1), (1,1), \infty\} \]
\[E_2(\mathbb{Q}) = \{(1,0), (0,-1), (0,1), \infty\}\]
None of the rational points yield a satisfying $L$.
\end{proof}

\begin{lem} \label{4t2}
Let $L \in \mathbb{Q} - \{0,1\}$. If $L^2 - L + 1 \in \mathbb{Q}^2$, then $L^2-L$ is not a square in $\mathbb{Q}$.
\end{lem}

\begin{proof}
Let $C$ be the curve defined by
\[C: y^2 = (x^2-x)(x^2-x+1).\]
A quick search yields the four points
\[C(\mathbb{Q})_{\mathrm{known}} = \{(0,0),(1,0), \infty, -\infty\},\]
none of which satisfies the conditions for $L$. It turns out that these are the only points of $C(\mathbb{Q})$. Indeed, as $C$ is of genus $1$ and has a point over $\mathbb{Q}$, we see that $C$ is an elliptic curve over $\mathbb{Q}$. It follows that the Jacobian $J_C$ of $C$ is an elliptic curve as well. Via Sage, we find that $J_C$ is isomorphic over $\mathbb{Q}$ to the curve 
\[J: y^2 = x^3 + \dfrac{2}{3}x + \dfrac{7}{27}.\]
The points $C(\mathbb{Q})$ are embedded into $J_C(\mathbb{Q})$ via the Abel-Jacobi map (see \ref{embedding} below), so
\[\#C(\mathbb{Q}) \leq \#J_C(\mathbb{Q}),\]
in the case both are finite.
However, as $J(\mathbb{Q})$ has exactly $4$ points, we find that $C(\mathbb{Q})$ is exactly the set of known points.
\end{proof}

\begin{prop} \label{4T} If $K/\mathbb{Q}$ is a quadratic extension, the torsion subgroup group $T_L(K)$ of $E_L(K)$ has a point of order $4$ exactly when $K=\mathbb{Q}(\sqrt{1-L})$.
\end{prop}
\begin{proof}
Let $(x,y)$ be a point of order order $4$ in $T_L(K)$. By the duplication formula (\cite{silverman1}, pp. 54), if $P=(x,y) \in T(K)$ is a point of order $4$, then the abscissa $x'$ of $2P$ must satisfy
\begin{equation} \label{dup}
x'= \dfrac{(x^2-L)^2}{4y^2} = \dfrac{(x^2-L)^2}{4x(x-1)(x-L)} = 0, 1, \text{ or } L.
\end{equation}
Solve for $x$ from equation (\ref{dup}) for each possible $x'$, we obtain
\begin{equation}
x = \pm \sqrt{L}, \, 1 \pm \sqrt{1-L}, \, L \pm \sqrt{L^2-L},
\end{equation}
none of which are rational by the Lemma \ref{4t1} and Lemma \ref{4t2}, which implies that none of the points of order $4$ are in $T_L(\mathbb{Q})$.

For the set of $x$-values above, the corresponding values of $y^2$ are
\begin{equation}
f(\pm \sqrt{L}) = -L(\sqrt{L} \mp 1)^2,
\end{equation}
\begin{equation}
f(1 \pm \sqrt{1-L}) =  (1-L)(1 \mp \sqrt{1-L})^2,
\end{equation}
\begin{equation}
f(L \pm \sqrt{L^2-L}) = (L-1)(L\pm \sqrt{L^2-L})^2.
\end{equation}

It follows that the points of order $4$ with the $x$-values $\pm \sqrt{L}$ or $L \pm \sqrt{L^2-L}$ all live in some bi-biquadratic extensions by Lemma \ref{4t1} and Lemma \ref{4t2}. The points of order $4$ that are in some quadratic extensions are exactly the ones with $x$-values $x = 1 \pm \sqrt{1-L}$, and all such points are in the quadratic extension $K = \mathbb{Q}(\sqrt{1-L})$.

\end{proof}

Now we will prove that $T_L(\mathbb{Q})$ does not have an element of order $3$. The proof requires Proposition \ref{hyper}, which we will prove later in Section \ref{proof} using Chabauty and Coleman's Method.

\begin{prop} \label{3T} The group $T_L(\mathbb{Q})$ does not have a point of order $3$.
\end{prop}
\begin{proof}
We will show that the $3$-divisional polynomial $\psi(x) = 3x^4 - 4(1+L)x^3 + 6Lx^2-L^2$ does not have a rational root. We will assume on the contrary that $\psi(x)$ has a rational root. Note that 
\[\psi\left(x+\frac{L+1}{3}\right) = 3x^4 - 2(L^2-L+1)x^2 -\frac{4}{9}(2L-1)(L+1)(L-2)x-\frac{1}{9}(L^2-L+1)^2.\]
Letting $B = (L^2-L+1)/x^2$ and $A = (2L-1)/x$, we obtain the cubic curve
\begin{equation} 
S: B^2 - 12AB + 18B = 27 - 4A^3.
\end{equation}
with a singularity at $(-3,9/2)$, where $A$ and $B$ also satisfy
\begin{equation}
B =(L^2-L+1)/x^2 \in \mathbb{Q}^2 \text{ and } \, A^2 - 4B = -3/x^2.
\end{equation}
The point $(-3,9/2)$ does not yield a rational $L$. Now, by \cite[Proposition 2.5]{silverman1}, as the singularity of $S$ is a cusp, the nonsingular rational points form a group and can be given by
\begin{equation}
\{(A,B) = (3 - t^2, 2t^3-6t^2+9): t \in \mathbb{Q}\}.
\end{equation}
The equation $A^2 - 4B = -3/x^2$ implies that $t$ must be of the form $t = (3-3u^2)/(1+3u^2)$ for some $u \in \mathbb{Q}$, and the fact that $B \in \mathbb{Q}^2$ implies that 
\begin{equation}
3u^6 + 75u^4 - 15u^2 + 1 = (3u^2+1)v^2
\end{equation}
for some $v \in \mathbb{Q}$.

Let $C$ be the projective curve of degree $6$ in $\mathbb{P}^2$ defined by
\begin{equation} \label{Ueq}
C: 3U^6 + 75U^4W^2 - 15U^2W^4 + W^6 - 3U^2V^2W^2-V^2W^4=0.
\end{equation}
The curve $C$ is birational to the hyperelliptic curve $H$ in the weighted projective plane $\mathbb{P}^2_3 = \mathbb{P}^2_{(1,4,1)}$ defined by
\begin{equation} \label{thecurve}
H: Y^2 = (X^2+XZ+Z^2)(X^6+3X^5Z-5X^3Z^2+3XZ^5+Z^6),
\end{equation}
where the birational transformation is given by coordinate mapping
\begin{equation}
X = U/2 + W/2, \, \quad Y = -3U^2VW/16 - VW^3/16, \quad \, Z = U/2 - W/2.
\end{equation}
As $U = X+Z$ and $W=X-Z$, we can easily obtain the $U$-coordinate of the finite points $(W=1)$ of equation \ref{Ueq} from Proposition \ref{hyper}:
\[U = \{1,1/3,0,-1,-1/3 \}.\]
Note that $U=0$ leads to $-3/X^2 = A^2 - 4B = 0$, which is not possible. All the other points lead to $L^2-L + 1 = -3B/(A^2-4B) = 1$, which is not possible since $L \neq 0,1$. Hence, the $3$-divisional polynomial $\psi(x)$ does not have a rational root.
\end{proof}

\begin{prop} If $E_L$ is a nice elliptic curve, then $E_L(\mathbb{Q})$ is isomorphic to $\mathbb{Z}/2\mathbb{Z} \oplus \mathbb{Z}/2\mathbb{Z}$.
\end{prop}
\begin{proof}
By Proposition \ref{4T}, $E_L(\mathbb{Q})$ does not have any point of order $4$, and by Proposition \ref{3T}, $E_L(\mathbb{Q})$ does not have a point of order $3$. Hence, the  proof of the proposition now follows.
\end{proof}

Below we will see that $E_L$ does not acquire more torsion points when extended to almost any quadratic extension over $\mathbb{Q}$.

Let $[D] \in \mathbb{Q}^{\times}/(\mathbb{Q}^{\times})2$ with $[D] \neq [1]$. For a nice elliptic curve $E_L: y^2 = f_L(x)$, we denote its quadratic twist 
\begin{equation}
E_L^D: y^2 = f_L^D(x)= x(x-D)(x-DL).
\end{equation}
Note also that the derivative of $f_L^D$ also split completely over $\mathbb{Q}$ as $L^2 - L + 1 \in \mathbb{Q}^2$.

\begin{rmk} If $E_L$ is a nice elliptic curve 
then we can twist $E_L$ to obtain curves of the form
\begin{equation}
E_{(a,b)}: y^2 = x(x-a)(x-b),
\end{equation}
where $a,b$ are distinct non-zero integers such that $a^2-ab+b^2$ is an integer square. In fact, using the technique from Proposition \ref{L}, we find that $a = u^2 - v^2$ and $b = 2uv - v^2$.
\end{rmk}

\begin{prop} \label{torsionoftwist} Let $E_L^D$ be a quadratic twist of a nice elliptic curve $E_L$. Then the torsion subgroup $T_L^D(\mathbb{Q})$ of $E_L^D(\mathbb{Q})$ is isomorphic to $\mathbb{Z}/2\mathbb{Z} \oplus \mathbb{Z}/2\mathbb{Z}$.
\end{prop}
\begin{proof}
We will show that $E_L^D(\mathbb{Q})$ does not have a point of order $4$ or a point of order $3$. The curves $E_L$ and its quadratic twist $E_L^D$ are isomorphic over $\mathbb{Q}(\sqrt{D})$ via the isomorphism given by
\begin{equation} \label{qiso}
E_L \rightarrow E_L^D, \hspace{0.2in} (x,y) \mapsto \left(Dx, D\sqrt{D}y\right).
\end{equation}
The isomorphism implies that since the the $3$-divisional polynomial $E_L$ does not have a $\mathbb{Q}$-rational root, then neither does the $3$-divisional polynomial of $E_L^D$. It follows that $E_L^D(\mathbb{Q})$ does not have a point of order $3$.

Lemma \ref{4T} implies that $E_L(K)$ has a point $(x,y)$ of order $4$ exactly when $K = \mathbb{Q}(\sqrt{1-L})$, specifically both $x$ and $y$ are in $\mathbb{Q}(\sqrt{1-L})$. From the isomorphism above, we see that $E_L^D$ has a point of order $4$ exactly when $Q(\sqrt{D}) = \mathbb{Q}(\sqrt{1-L})$.

\end{proof}

\begin{cor} Let $E_L$ be a nice elliptic curve with a quadratic twist $E_L^D$. If $K/\mathbb{Q}$ is a quadratic extension, then $E_L^D(K)$ has a point of order $4$ exactly when $K = \mathbb{Q}(\sqrt{1-L}) = \mathbb{Q}(\sqrt{D})$.
\end{cor}

 The lemma below is well-known (see \cite[Corollary 4]{gj} and \cite[Lemma 1.1]{ll}).
For a general elliptic curve $E: y^2 = f(x)$, we will denote its quadratic twist by
\begin{equation}
E^D: Dy^2 = f(x).
\end{equation}
\begin{lem} \label{twist}Let $E/\mathbb{Q}$ be an elliptic curve and $E^D$ its quadratic twist. We have the following isomorphism for odd integer $n$:
\[E(\mathbb{Q}(\sqrt{D}))[n] \simeq E(\mathbb{Q})[n] \oplus E^D(\mathbb{Q})[n].\]
\end{lem}

\begin{cor} \label{3T2} If $E_L$ is a nice elliptic curve, then $E_L(K)[3]$ is trivial for all quadratic extension $K/\mathbb{Q}$.
\end{cor}
\begin{proof}
As $E_L(\mathbb{Q}[3]$ and $E^D_L(\mathbb{Q})[3]$ are both trivial, the corollary follows from Lemma \ref{twist}.
\end{proof}

\noindent \textbf{Proof of Theorem \ref{torsion}} 
\begin{proof}
The proof of Theorem \ref{torsion} now follows. Let $K/\mathbb{Q}$ be a quadratic extension. The group $E_L(K)$ already contains the Klein four group. By Proposition \ref{4T}, the group $E_L(K)$ has a point of order $4$ exactly when $K =  \mathbb{Q}(\sqrt{1-L})$. By Corollary \ref{3T2}, the group $E_L(K)$ does not have a point of order $3$. Hence, the theorem is proven by Mazur's Theorem.
\end{proof}

In the following, we will look at points of $E_L(\mathbb{Q})$ of infinite order. As the torsion points  of $E_L$ over quadratic extensions $K/\mathbb{Q}$ are explicitly given, points of infinite order can be readily found. Indeed, the following proposition can be easily verified.

\begin{prop} If $L$ is defined by 
\[L = \dfrac{(u^2+3)(u^2-1)}{4u^2},\]
where $u \in \mathbb{Q} - \{0, \pm 1\}$, then $(s, \sqrt{f_L(s)})$ is a rational point of infinite order of the nice elliptic curve $E_L$.
\end{prop}

\begin{prop} The points $P=(r,\sqrt{f(r)})$ and $Q=(s,\sqrt{f(s)})$ are linearly independent points of infinite order of $E_L(K)$, where $K = \mathbb{Q}(\sqrt{-3})$.
\end{prop}
\begin{proof}
Since the torsion subgroup of $E_L(K)$ is isomorphic to $\mathbb{Z}/2\mathbb{Z} \oplus \mathbb{Z}/2\mathbb{Z}$ by the corollary above and $r,s \neq 0,1,$ or $L$, it follows that $P$ and $Q$ are indeed points of infinite order. 

Now, suppose on the contrary that $P$ and $Q$ are linearly dependent. There exist integers $n$ and $m$ such that $mP + nQ = 0$. Let $\sigma$ be the nontrivial element of $\mathrm{Gal}(K/\mathbb{Q})$. It can be seen that $\sigma(Q) = -Q$. Then $mP + nQ = 0$ implies that $mP - nQ = 0$, which is a contradiction as it would imply that $P$ has finite order.
\end{proof}

\begin{cor} For any nice elliptic curve $E_L$, the rank of $E_L(\mathbb{Q}(\sqrt{-3}))$ is at least $2$, with generators $(r, \sqrt{f_L(r)})$ and $(s,\sqrt{f(s)})$.
\end{cor}

\section{Points on hyperelliptic curves with Chabauty
} \label{chab}

In this section, we will begin with providing a short overview of hyperelliptic curves and Chabauty-Coleman's Method. Our brief overview follows from \cite{cohen2}, \cite{jennifer}, \cite{mp}, and \cite{stoll}.

Faltings shows that if $C/\mathbb{Q}$ is a smooth projective curve of genus $g:= \mathrm{genus}(C) \geq 2$, then the set $C(\mathbb{Q})$ of rational points of $X$ is finite \cite{faltings}. Points in $C(\mathbb{Q})$ can be obtained by naively searching for points on its affine singular model $f(x,y) = 0$. The naive search should be quick as the size of the rational points is expected to be small with respect to the coefficients of $f(x,y)$ \cite{ih}, unlike in the genus $1$ case. Of course, the difficulty lies in proving that we have the complete set of rational points.

We say that $C$ is hyperelliptic if there exists a degree-$2$ map $C \rightarrow \mathbb{P}^1$ and $g(C) \geq 2$ (following \cite{hindryandsilverman}). In particular, if $f(x) \in \mathbb{Q}[x]$ is separable, then the affine equation $C: y^2 = f(x)$ defines a hyperelliptic curve. The projective closure of $X$ in $\mathbb{P}^2$ is singular at the infinite points. Its smooth projective model is embedded in $\mathbb{P}^{g(C)+2}$ (see \cite[Exercise II.2.14]{silverman1} and \cite[Section A.4.5]{hindryandsilverman}), which is not practical in the search for rational points. On the other hand, the homogenization $Y^2 = Z^{2g+2}f(X/Z)$ is a smooth model of $C$ in the weighted projective plane $\mathbb{P}_{1,g+1,1}$, where the variables $X,Y$, and $Z$ have weights $1,g+1$, and $1$, respectively.

By \cite[Theorem II.2.4(c)]{silverman1}, for each subgroup $G$ of $\mathrm{Aut}(C/\mathbb{Q})$, there exists a smooth curve $C_G:= G\backslash C$ over $\mathbb{Q}$ and a morphism $\phi: C \rightarrow C_G$ defined over $\mathbb{Q}$ of degree $[\mathrm{Aut}(H):G]$ such that $\phi^*\mathbb{Q}(H) = \mathbb{Q}(C_G)$, where $\phi^*$ is the pullback map. Here, we say that $C_G$ is the quotient of $H$ by the subgroup $G$. Via the quotient map
\begin{equation}
\phi: C \rightarrow C_G,
\end{equation}
one can compute the points of $C(\mathbb{Q})$ using the pullback $\phi^{-1}(C_G(\mathbb{Q}))$ if $C_G(\mathbb{Q})$ can be computed. In particular, we can work with a curve $C_g$ of genus smaller than that of $C$.

As the set of points of $C(\mathbb{Q})$ does not have a geometric group structure, we use the Abel-Jacobi embedding 
\begin{equation} \label{embedding}
 \iota: C(\mathbb{Q}) \hookrightarrow J(\mathbb{Q}), \hspace{0.2in} P \mapsto [(P)-(P_0)],
 \end{equation}
where $J$ is the Jacobian of $C$ and $P_0$ is a point of $C(\mathbb{Q})$. Here we identify the Jacobian $J(\mathbb{Q})$ with the Picard group $\mathrm{Pic}^0(C/\mathbb{Q})$. By Mordell-Weil Theorem, $J(\mathbb{Q})$ is finitely-generated. In the fortunate case that $J(\mathbb{Q})$ is finite, each point of $[D] \in J(\mathbb{Q})$ can be pulled back to points $P \in C(\mathbb{Q})$ by $(P) = D + (P_0) + \mathrm{div}(f)$, for non-zero functions $f$ with $\mathrm{div}(f) \geq -D - (P_0)$.

When the rank $r$ of $J(\mathbb{Q})$ is non-zero but is less than the genus $g$ of $C$, we can apply the Chabauty's Method. For a prime $p > 2$ of good reduction for $C$, the Abel-Jacobi embedding $\iota$ in $(\ref{embedding})$ induces an isomorphism of $\mathbb{Q}_p$-vector spaces of translation-invariant, holomorphic differentials of dimension $g$ \cite[Proposition 2.2]{milne}:
\begin{equation} \label{cotangent}
H^0(J(\mathbb{Q}_p), \Omega^1) \rightarrow H^0(C(\mathbb{Q}_p), \Omega^1), \hspace{0.2in} \omega_J \mapsto \omega:= \omega_J \vert_{C(\mathbb{Q}_p)},
\end{equation}
which yields isomorphism of tangent spaces of dimension $g$ over $\mathbb{Q}_p$:
\begin{equation} \label{tangent}
H^0(J(\mathbb{Q}_p), \Omega^1)^* \rightarrow H^0(C(\mathbb{Q}_p), \Omega^1)^*.
\end{equation}
Via (\ref{cotangent}), for each $\int \omega_J \in H^0(J(\mathbb{Q}_p), \Omega^1)^*$, there exists $\int \omega \in H^0(C(\mathbb{Q}_p), \Omega^1)^*$ such that
\begin{equation}
\int_{P_0}^P \omega = \int_0^{[(P)-(P_0)]} \omega_J.
\end{equation}
Now from the following chain of maps
\begin{equation}
C(\mathbb{Q}) \xhookrightarrow{P \mapsto [(P)-(P_0)]} J(\mathbb{Q}) \xhookrightarrow{} J(\mathbb{Q}_p) \xrightarrow{D \mapsto \int_D} H^0(J(\mathbb{Q}_p), \Omega^1)^* \rightarrow H^0(C(\mathbb{Q}_p), \Omega^1)^*,
\end{equation}
it follows that $C(\mathbb{Q})$ is mapped into a subspace of $H^0(C(\mathbb{Q}_p), \Omega^1)^*$ of dimension at most $r$.  Hence, as $g - r > 0$, there exists an \textit{annihilating differential} $\omega_A \in H^0(J(\mathbb{Q}_p), \Omega^1)$ such that $\int_0^D \omega_A = 0$ for all $D \in J(\mathbb{Q})$.
Indeed, it can be seen that the integrals $\int \omega_J$ are equal to $0$ on the torsion points of $J(\mathbb{Q}_p)$. Let $Q_1, Q_2, \ldots, Q_r$ be a set generators of the free part of $J(\mathbb{Q})$, and let $\omega_1, \ldots, \omega_g$ be a basis for $H^0(J(\mathbb{Q}_p), \Omega^1)$. The dimension of the null space of the matrix $M$ is at least $g-r > 0$, where $M$ is defined by
\begin{equation}
M = \left(\int_0^{Q_i} \omega_j \right)_{1 \leq i \leq r, 1 \leq j \leq g}.
\end{equation}
Hence, there exists a non-trivial vector $(c_1, c_2, \ldots, c_g) \in \mathbb{Q}_p^g$ such that for each $1 \leq i \leq r$, we have
\begin{equation}
c_1 \int_0^{Q_i} \omega_1 + c_2 \int_0^{Q_i} \omega_2 + \dots + c_g\int_0^{Q_i} \omega_g = 0.
\end{equation}
By linearity, $\int_0^D \omega_A = 0$ for all $D \in J(\mathbb{Q})$, where
\begin{equation} \label{ann}
\omega_A = c_1 \omega_1 + c_2 \omega_2 + \dots + c_g \omega_g.
\end{equation}
As $C(\mathbb{Q})$ is embedded in $J(\mathbb{Q})$, the claim above follows.

Via $(\ref{tangent})$, we can explicitly compute $\omega_A$ over $C(\mathbb{Q}_p)$, and integrals over $C(\mathbb{Q}_p)$ can be easily calculated by using a linear interpolation from $P$ to $Q$:
\begin{equation}
\int_{P}^Q \omega = \int_0^1 \omega(x(t),y(t)) \, dt,
\end{equation}
which is convergent if it's a \textit{tiny integral}, that is, if $P \equiv Q \pmod{p}$. Note that here $t$ is also a uniformizer at $P$.  Assume $C$ is of the form
\begin{equation}
C: y^2 = f(x),
\end{equation}
in which case, a set of basis for $H^0(C(\mathbb{Q}_p), \Omega^1)$ is
\begin{equation}
\dfrac{dx}{2y}, \, \dfrac{xdx}{2y}, \, \dfrac{x^2 dx}{2y}, \, \dots, \, \dfrac{x^{g-1}dx}{2y}.
\end{equation}
Let $P_0, \ldots, P_k$ be the set of known points of $C(\mathbb{Q})$. Any non-trivial vector $(c_1, c_2, \ldots, c_g)$ in $\mathrm{Nul}(M)$ will give us an annihilating differential $\omega_A$, where
\begin{equation}
M = \left(\int_{P_0}^{P_i} \dfrac{x^j dx}{2y} \right)_{1 \leq i \leq k, 1 \leq j \leq g}.
\end{equation}
By reordering if necessary, choose a subset $S=\{P_0, P_1, \ldots, P_n\}$ of $\{P_0,P_1, \ldots, P_k\}$ that are mapped bijectively to $C(\mathbb{F}_p)$. Finally, the set $C(\mathbb{Q})$ will be a subset of the zeroes for the function
\begin{equation}
f(z) = \int_{P_i}^z \omega_A
\end{equation}
for each point $P_i$ of $S$.

\section{Points on the hyperelliptic curve $H$} \label{proof}

Recall the curve $H$ given in $(\ref{thecurve})$. In this section we will show that $H(\mathbb{Q})$has exactly $12$ points. All computations are in this section are done with the aid of MAGMA.

Using MAGMA, we find that the automorphism group Aut$(H)$ is isomorphic to the dihedral group $D_{12}$ of order $12$. Each of the $7$ elements $g$ of order $2$ in Aut$(H)$ corresponds to an extension of function fields $\mathbb{Q}(H)/K_g$ of degree $2$. There exists a smooth curve $H_g/\mathbb{Q}$ and a morphism $\phi_g: H \rightarrow H_g$ of degree $2$ defined over $\mathbb{Q}$ such that $\phi_g^*\mathbb{Q}(H') = K_g$, where $\phi_g^*$ is the pullback map. Here, we say that $H_g$ is the quotient of $H$ by the subgroup generated by $g$. Three of the $7$ elements of order $2$ yield a quotient $\phi_g: H \rightarrow H_g$ with $H_g$ having genus $2$. In fact, each the three morphisms $\phi_g$ is an $\acute{e}$tale cover by the Riemann-Hurwitz formula \cite[Theorem II.5.9]{silverman1} as it is an unramified morphism of smooth curves. In particular we have the $\acute{e}$tale double cover 
\begin{equation}
\varphi: H \rightarrow H_q,
\end{equation}
where $H_q$ is defined by
\begin{equation}
H_q: Y^2 = 4X^5 - 7X^4Z - 2X^3Z^2 + 16X^2Z^3 + 8XZ^4 + Z^5,
\end{equation}
and the morphism $\varphi$ is defined by
\begin{equation} \label{map}
\varphi(X,Y,Z) = (XZ^5,X^2YZ^{12}-YZ^{14},-X^2Z^4-2XZ^5-Z^6).
\end{equation}

\begin{lem} There exists an $\acute{e}$tale double cover $\varphi: H \rightarrow H_q$, where $H_q$ is a hyperelliptic curve of genus $2$.
\end{lem}

The hyperelliptic curve $H_q$ has genus $2$ and its Jacobian $J:= \mathrm{Jac}(H_q)$ has rank $1$, so the Chabauty-Coleman's Method is applicable. The point $P:= [(0,-1) - (\infty)] \in J$ has infinite order. Applying the Magma function \verb+Chabauty(P)+, we obtain Lemma \ref{quotient}. Pulling the points of $H_q$ back to the points of $H$ via the map in \ref{map}, we obtain Proposition \ref{hyper}.

\begin{lem} \label{quotient} Let $H_q$ be the hyperelliptic curve defined by its affine model
\[H_q: y^2 = 4x^5 - 7x^4 - 2x^3 + 16x^2 + 8x + 1.\]
The complete set of rational points of $H_q$ is the $7$ points
\[\{\infty, (0,\pm 1), (-1,0), (-1/4,0), (2,\pm 9) \}.\]
\end{lem}

\begin{prop} \label{hyper} Let $H$ be the hyperelliptic curve defined by its affine model
\[H: y^2 = (x^2+x+1)(x^6+3x^5-5x^3+3x+1).\]
The complete set of rational solutions in the weighted projected plane $\mathbb{P}_3^2(\mathbb{Q})$ are the $12$ points
\[\scalemath{0.9}{\{\pm \infty, (2/3,\pm 1/27,-1/3),(1/2,\pm 1/16,-1/2),(0, \pm 1,-1), (1,\pm 3,1),(1/3,\pm 1/27,-2/3)\}.}\]
\end{prop}

\section{Some words on Galois representation} \label{galois}
Let $E_L$ be a nice elliptic curve defined by $E_L: y^2 = f(x)$ for some nice cubic polynomial with distinct roots of the form $f(x) = x(x-a)(x-b)$. Let $L = b/a$. The elliptic curve $E_L$ defined by
\[E_L: y^2 = x(x-1)(x-L)\]
is $K$-isomorphic to $E_L$ for some finite extension $K/\mathbb{Q}$ of degree at most $2$. The $j$-invariant of $E_L$ is
\begin{equation}
j(E_L) = 16^2 \dfrac{(L^2-L+1)^3}{(L^2-L)^2},
\end{equation}
which is a nonzero element of $\mathbb{Q}^2$ as $L$ is a rational number not equal to $0$ or $1$, and $L^2-L + 1$ is a perfect square. Since the elliptic curves in each $\mathbb{C}$-isomorphism class have the same $j$-invariant, we find that $j(E_L) = j(E^D_L)$. Among the $13$ isomorphism classes of elliptic curves over $\mathbb{Q}$, only one has $j$-invariant in $\mathbb{Q}^2$, namely, $j=0$ (\cite{silverman2}, Appendix A and \cite{cohen}, pp. 383). Hence, it follows that $E_L$ does not complex multiplication for any nice cubic polynomials $f(X)$ with distinct roots.

Let $G$ be the absolute Galois group Gal$(\overline{\mathbb{Q}}/\mathbb{Q})$, where $\overline{\mathbb{Q}}$ is the algebraic closure of $\mathbb{Q}$. Let $\rho_{f,\ell}: G \rightarrow \mathrm{Aut}(T_{\ell}(E_L)$ be the $\ell$-adic representation associated to $E_L$. It is known that the representation $\rho_{f,\ell}$ is surjective for all except perhaps for finitely many primes $\ell$ \cite{serre}. There is data suggesting that $\rho_{f,\ell}$ is surjective for all odd primes $\ell$.

\section*{Acknowledgment}
I am grateful to Jeremy Rousse for suggesting me to use MAGMA \cite{magma} and for explaining certain parts on elliptic curves.

\section*{Appendix}
\begin{verbatim}
P<X>:=PolynomialRing(Rationals());
C:=HyperellipticCurve((X^2 + X + 1)*(X^6 + 3*X^5 - 5*X^3 + 3*X + 1));
AutC:=Automorphisms(C); AutC[11];
G:=AutomorphismGroup(C,[AutC[11]]);
Cg, phi :=CurveQuotient(G); Cg;
Points:=Points(Cg : Bound:=1000);
for P in Points do
    preimageofP:= P @@ phi;
    RationalPoints(preimageofP);
end for; 
\end{verbatim}

\bibliography{ref}

\begin{thebibliography}{{The}21}

\bibitem[Bal11]{jennifer}
Jennifer~Sayaka Balakrishnan.
\newblock {\em Coleman Integration for Hyperelliptic Curves: Algorithms and
  Applications}.
\newblock PhD thesis, Massachusetts Institute of Technology, 2011.

\bibitem[BCP97]{magma}
Wieb Bosma, John Cannon, and Catherine Playoust.
\newblock The {M}agma algebra system. {I}. {T}he user language.
\newblock {\em J. Symbolic Comput.}, 24(3-4):235--265, 1997.
\newblock Computational algebra and number theory (London, 1993).

\bibitem[BFM92]{bd}
Jim Buddenhagen, Charles Ford, and Mike May.
\newblock Nice cubic polynomials, pythagorean triples, and the law of cosines.
\newblock {\em Mathematics Magazine}, 65, No. 4:244--249, 1992.

\bibitem[Cho15]{ajai}
Ajai Choudhry.
\newblock A diophantine problem from calculus.
\newblock {\em Journal of Number Theory}, 153:354--363, 2015.

\bibitem[Coh93]{cohen}
Henri Cohen.
\newblock {\em A course in computation algebraic number theory}, volume 138 of
  {\em Graduate Texts in Mathematics}.
\newblock Springer-Verlag, Berlin Heidelberg, 1993.

\bibitem[Coh07]{cohen2}
Henri Cohen.
\newblock {\em Number Theory, Volume II: Analytic and Modern Tools}, volume 240
  of {\em Graduate Texts in Mathematics}, pages 452--462.
\newblock Springer-Verlag, New York, 2007.

\bibitem[DS05]{diamondandshurman}
Fred Diamond and Jerry Shurman.
\newblock {\em A First Course in Modular Forms}.
\newblock Springer, New York, NY, 2005.

\bibitem[Fal83]{faltings}
G.~Faltings.
\newblock Endlichkeitssätze für abelsche varietäten über zahlkörpern.
\newblock {\em Inventiones mathematicae}, 73:349--366, 1983.

\bibitem[GJT14]{gj}
Enrique Gonz{\'a}lez-Jim{\'e}nez and Jos{\'e}~M. Tornero.
\newblock Torsion of rational elliptic curves over quadratic fields.
\newblock {\em Revista de la Real Academia de Ciencias Exactas, Fisicas y
  Naturales. Serie A. Matematicas}, 108(2):923--934, Sep 2014.

\bibitem[HS00]{hindryandsilverman}
Marc Hindry and Joseph~H. Silverman.
\newblock {\em Diophantine Geometry: An Introduction}, volume 201 of {\em
  Graduate Texts in Mathematics}.
\newblock Springer-Verlag, New York, 1st edition, 2000.

\bibitem[Ih02]{ih}
Su-Ion Ih.
\newblock Height uniformity for algebraic points on curves.
\newblock {\em Composito Mathematica}, 134:35--57, 2002.

\bibitem[LL85]{ll}
Michael Laska and Martin Lorenz.
\newblock Rational points on elliptic curves over $\mathbb{Q}$ in elementary
  abelian 2-extension of $\mathbb{Q}$.
\newblock {\em J. Reine Angew. Math.}, 355:163--172, 1985.

\bibitem[Mil86]{milne}
J.~S. Milne.
\newblock Jacobian varieties.
\newblock In Gary Cornell and Joseph~H. Silverman, editors, {\em Arithmetic
  Geometry}, pages 167--212. Springer, New York, 1986.

\bibitem[MP12]{mp}
William McCallum and Bjorn Pooen.
\newblock The method of chabauty and coleman.
\newblock In {\em Explicit methods in number theory : Rational points and
  diophantine equations}, pages 99--117. Panoramas et synth{\'e}ses -
  Soci{\'e}t{\'e} math{\'e}matique de France, Paris, 2012.

\bibitem[Ser72]{serre}
Jean-Pierre Serre.
\newblock Propri$\acute{e}$t$\acute{e}$s galoisiennes des points d’ordre fini
  des courbes elliptiques.
\newblock {\em Inventiones Mathematica}, 15(4):259--331, 1972.

\bibitem[Sil94]{silverman2}
Joseph~H. Silverman.
\newblock {\em Advanced Topics in the Arithmetic of Elliptic Curves}, volume
  151 of {\em Graduate Texts in Mathematics}.
\newblock Springer-Verlag, New York, 2nd edition, 1994.

\bibitem[Sil09]{silverman1}
Joseph~H. Silverman.
\newblock {\em The arithmetic of elliptic curves}, volume 106 of {\em Graduate
  Texts in Mathematics}.
\newblock Springer-Verlag, New York, 2009.

\bibitem[Sto14]{stoll}
Michael Stoll.
\newblock Arithmetic of hyperelliptic curves, Summer 2014.

\bibitem[{The}21]{sagemath}
{The Sage Developers}.
\newblock {\em {S}ageMath, the {S}age {M}athematics {S}oftware {S}ystem
  ({V}ersion 9.4)}, 2021.
\newblock {\tt https://www.sagemath.org}.

\end{thebibliography}

\end{document}